\documentclass[10pt]{amsart}

\usepackage{graphicx}

\usepackage[USenglish]{babel}
\usepackage{latexsym}
\usepackage{amsmath}
\usepackage{amsfonts}
\usepackage{amssymb}
\usepackage{color}
\usepackage{esint}
\usepackage[all]{xy}
\renewcommand{\a }{\alpha }
\renewcommand{\b }{\beta }
\renewcommand{\d}{\delta }
\newcommand{\D }{\Delta }

\newcommand{\e }{\varepsilon }
\newcommand{\g }{\gamma}

\renewcommand{\l }{\lambda }

\newcommand{\n }{\nabla }
\newcommand{\var }{\varphi }
\newcommand{\rh }{\rho }

\newcommand{\s }{\sigma }

\renewcommand{\t }{\tau }
\renewcommand{\th }{\theta }

\renewcommand{\O }{\Omega }

\newcommand{\ov}{\overline}

\def\p{\partial}

\newcommand{\wtilde }{\widetilde}

\newcommand{\be}{\begin{equation}}
\newcommand{\ee}{\end{equation}}

\newtheorem{remark}{Remark}[section]
\newcommand{\R}{\mathbb{R}}

\newcommand{\N}{\mathbb{N}}

\newcommand{\dis}{\displaystyle}

\newcommand{\dkr}{{\bf d}}

\newtheorem{theorem}{Theorem}[section]
\newtheorem{proposition}[theorem]{Proposition}
\newtheorem{example}[theorem]{Example}

\newcommand{\bpr}{\begin{proposition}}
\newcommand{\epr}{\end{proposition}}
\newcommand{\bex}{\begin{example}\rm}
\newcommand{\eex}{\end{example}}

\begin{document}

\newtheorem{lem}{Lemma}[section]
\newtheorem{pro}[lem]{Proposition}
\newtheorem{thm}[lem]{Theorem}
\newtheorem{rem}[lem]{Remark}
\newtheorem{cor}[lem]{Corollary}
\newtheorem{df}[lem]{Definition}

\title[Mean field equation involving probability measures]{A mean field equation involving positively supported probability measures: blow-up phenomena and variational aspects}

\author{Aleks Jevnikar, Wen Yang}

\address{ Aleks Jevnikar,~University of Rome `Tor Vergata', Via della Ricerca Scientifica 1, 00133 Roma, Italy}
\email{jevnikar@mat.uniroma2.it}

\address{ Wen ~Yang,~Center for Advanced Study in Theoretical Sciences (CASTS), National Taiwan University, Taipei 10617, Taiwan}
\email{math.yangwen@gmail.com}




\begin{abstract}
We are concerned with an elliptic problem which describes a mean field equation of the equilibrium turbulence of vortices with variable intensities.

In the first part of the paper we describe the blow-up phenomenon and highlight the differences from the standard mean field equation.

In the second part we discuss the Moser-Trudinger inequality in terms of the blow-up masses and get the existence of solutions in a non-coercive regime by means of a variational argument, which is based on some improved Moser-Trudinger inequalities.
\end{abstract}

\maketitle

\section{Introduction}

\medskip

We are concerned with the following equation:
\begin{equation} \label{eq}
  - \D u = \rho_1 \left( \frac{h_1 e^{u}}{\int_M
      h_1 e^{u} \,dV_g} - \frac{1}{|M|} \right) + a\rho_2 \left( \frac{h_2 e^{au}}{\int_M
      h_2 e^{au} \,dV_g} - \frac{1}{|M|} \right),
\end{equation}
where $h_1, h_2$ are smooth positive functions, $\rho_1, \rho_2$ are two positive parameters, $a\in(0,1)$  and $M$ is a compact orientable surface with Riemannian metric $g$. For simplicity we always assume the total volume of $M$ is $|M|=1$.

Equation \eqref{eq} is motivated by turbulence flows and it was first proposed by Onsager \cite{ons}. More precisely, in case the circulation number density is ruled to a probability measure, under a \emph{deterministic} assumption on the point vortex intensities, see for example \cite{sa-su}, the model is the following:
\begin{equation} \label{eq:prob}
  - \D u = \rho \int_{[-1,1]} \a \left( \frac{ e^{\a u}}{\int_M
      e^{\a u} \,dV_g} - \frac{1}{|M|} \right) \mathcal P(d\a),
\end{equation}
where $u$ corresponds to the stream function of a turbulent flow, $\mathcal P$ is a Borel probability measure defined on the interval $[-1,1]$ describing the point vortex intensity distribution and $\rho>0$ is a constant related to the inverse temperature. We restrict our discussion to the choice $\mathcal P(d\a)= \t \d_1(d\a)+(1-\t) \d_{a}(d\a)$, where $a\in(0,1)$ and $\t_1\in[0,1]$. Moreover, we will see it differs from the different-sign case (i.e. $a<0$) not only on the study of the blow-up local masses but also on the associated Moser-Trudinger inequality.

\medskip

When $\mathcal P(d\a) = \d_1$, equation \eqref{eq:prob} turns to be the standard mean field equation
\begin{equation} \label{eq:liouv}
  - \D u = \rho\left( \frac{h\, e^{u}}{\int_M  h\, e^{u} \,dV_g}- \frac{1}{|M|} \right).
\end{equation}
The latter equation is motivated both in geometry and mathematical physics as it is related to the prescribed Gaussian curvature problem and the mean field equation of Euler flows. The equation \eqref{eq:liouv} has attracted a lot of attentions in the past decades and we refer the readers to \cite{ba, cha, cha2, clmp, ki, ly, mal, s, tar}.

The study of \eqref{eq:liouv} is well-understood now. For our purpose we shall state the result on the quantization obtained in \cite{bm, li, li-sha}. For a sequence of solutions $\{u_k\}$ to \eqref{eq:liouv} which blows-up at $\bar x$, we have
\begin{equation} \label{liouv:quant}
\lim_{\d\to 0} \lim_{k\to+\infty} \rho_k \frac{\int_{B_\d(\bar x)}h\, e^{u_k}}{\int_M  h\, e^{u_k} \,dV_g} = 8\pi.
\end{equation}
A direct consequence of the above quantization result is the blow-up phenomenon can occur only for $\rho\in8\pi\N$.

Recently, the latter result was extended to the symmetric case in \cite{jwy}, where $\mathcal P(d\a)= \t \d_1(d\a)+(1-\t) \d_{-1}(d_{-1})$. It is known as the sinh-Gordon equation (see \cite{je-ya} for an analogous problem):
\begin{equation} \label{eq:sinh}
  - \D u = \rho_1 \left( \frac{h_1 e^{u}}{\int_M
      h_1 e^{u} \,dV_g} - \frac{1}{|M|} \right) - \rho_2 \left( \frac{h_2 e^{-u}}{\int_M
      h_2 e^{-u} \,dV_g} - \frac{1}{|M|} \right),
\end{equation}
In the asymmetric case of \eqref{eq} there are by now just partial results, see for example \cite{os, os1, ors, pi-ri, ri-ze}.

We treat the problem \eqref{eq} by following the ideas in \cite{jwy} which is concerned for the equation \eqref{eq:sinh} (this idea was first introduced in \cite{lin-wei-zh} for the $SU(3)$ Toda system).  However, the situation is quite different for the same-sign case, where we may meet a new blow-up phenomena as highlighted in Theorem \ref{th:blow-up} and in Remark \ref{rem:blow-up}. Before we state the results, let us first fix some notation for later use. Since the problem \eqref{eq} is invariant by translation we will often work in the following space:
$$
	\mathring H^1(M) = \left\{ u\in H^1(M) \; : \; \int_M u = 0 \right\}.
$$
For a sequence of solutions $u_k$ to \eqref{eq} as $\rho_{i,k} \to  \rho_i$, $i=1,2$, we consider the normalized functions
\begin{align} \label{norm}
u_{1,k} = u_k - \log \int_M h_{1} \,e^{u_{k}} \,dV_g, \qquad u_{2,k} = a u_k - \log \int_M h_{2} \,e^{a u_{k}} \,dV_g,
\end{align}
which satisfy
$$
	-\D u_{1,k} = \rho_{1,k} \left(h_1 e^{u_{1,k}}-1\right) + a\rho_{2,k} \left( h_2 e^{u_{2,k}}-1\right),
$$
and we define the blow-up set $S=S_1\cup S_2$, where
\begin{equation} \label{blow-up set}
S_i= \biggr\{ p\in M \,:\, \exists\{x_k\}\subset M, \, x_k\to p, \, u_{i,k}(x_k) \to +\infty \biggr\}, \quad i=1,2.
\end{equation}
The first result is the following
\begin{theorem} \label{th:blow-up}
Let $u_k \in \mathring H^1(M)$ be a sequence of solutions to \eqref{eq} with $a\in (0,1)$ and let $u_{i,k}$, $S$ be defined as above. Then, by passing to a subsequence if necessary, the following alternatives hold true:
\begin{enumerate}
	\item (compactness) $S=\emptyset$ and $u_k$ is uniformly bounded in $L^{\infty}(M)$.
	
	\item (blow-up) $S\neq\emptyset$ and it is finite. It holds
	\begin{equation} \label{limit}
		\rho_{1,k} h_1 e^{u_{1,k}}\rightharpoonup r_1 +  \sum_{p\in S_1} m_1(p) \d_p\,, \qquad  \rho_{2,k} h_2 e^{u_{2,k}} \rightharpoonup r_2 + \sum_{p\in S_2} m_2(p) \d_p\,,
	\end{equation}
	in the sense of measures, where $r_i \in L^1(M)\cap L^{\infty}_{loc}(M\setminus S_i),~i=1,2$ and
	$$
		m_i(p) = \lim_{r\to 0} \lim_{k\to+\infty} \rho_{i,k} \int_{B_r(p)} h_i e^{u_{i,k}} \,dV_g.
	$$
	Moreover, we have:
	
	\medskip
	
	\begin{itemize}
	\item[(a)] If $a\in\left[\frac 12,1\right)$, then $(m_1(p),m_2(p))$ is one of the following types:
	\begin{equation} \label{mass1}
	(8\pi,0), \quad \left(0,\dfrac{8\pi}{a^2}\right), \quad (m_1,m_2),
	\end{equation}
	where $m_1\in(0,8\pi)$, $m_2\in\left(0, \dfrac{8\pi}{a^2} \right)$. In addition $(m_1,m_2)$ satisfies 
	$$m_1+am_2 >  8\pi \quad \mbox{and} \quad (m_1+a m_2)^2=8\pi(m_1+m_2).$$
	
	\
	
	\item[(b)] If $a\in\left(0,\frac 12\right)$, then $(m_1(p),m_2(p))$ is one of the following types:
	\end{itemize}
	\begin{equation} \label{mass2}
	(8\pi,0), \quad \left(0,\dfrac{8\pi}{a^2}\right), \quad \left(8\pi,\dfrac{8\pi}{a^2}-\dfrac{16\pi}{a}\right),
    \quad (m_1,m_2),  \quad (8\pi m, \eta_m),
	\end{equation}
	where $m_1\in(0,8\pi)$, $m_2\in\left(\dfrac{8\pi}{a^2}-\dfrac{16\pi}{a}, \dfrac{8\pi}{a^2} \right)$ and
    $(m_1,m_2)$ satisfies
$$m_1+am_2 >  \dfrac{8\pi}{a}-8\pi \quad \mathrm{and} \quad (m_1+a m_2)^2=8\pi(m_1+m_2),$$
while $m\in\N$, $m>1$ and $\eta_m$ satisfies $(8\pi m+a\eta_m)^2=8\pi(8\pi m+\eta_m)$.
\end{enumerate}
\end{theorem}

\medskip

We list here some remarks concerning the latter result.

\begin{rem} \label{rem:blow-up}
Theorem \ref{th:blow-up} exhibits drastic differences from the standard mean field equation \eqref{eq:liouv} and the opposite-sign case, see for example \eqref{eq:sinh}. Indeed, the local blow-up masses are not quantized any more as we have the type $(m_1,m_2)$ in \eqref{mass1}, \eqref{mass2} for equation \eqref{eq}. This phenomenon is due to the same-sign structure: the reason is that in the opposite-sign case the energy limit is given by a globally defined solution of $-\D v = e^v$ (or $-\D v = a^2e^{-av}$), while our problem could exhibit a \underline{fully blow-up} phenomenon which may lead us to consider an equation of the form $-\D v = e^v + a e^{av}$: the latter equation was recently considered in \cite{tar2} and gives rise to the type $(m_1,m_2)$ in \eqref{mass1}, \eqref{mass2}.
\end{rem}

\begin{rem} \label{rem:min-mass}
From the description in \eqref{mass2} if $a\in\left(0,\frac 12\right)$ and $m_2(p)\neq 0$ we could get a minimum mass. This in turn gives us some partial information about the boundedness of the solutions to \eqref{eq}, see Theorem \ref{th:comp} (see also the existence result of Theorem \ref{th:exist}). On the other hand, for $a\in\left[\frac 12,1\right)$ a minimum mass is not guaranteed and we can not derive either boundedness properties or the existence of solutions to \eqref{eq}.
\end{rem}

\begin{rem} \label{rem:resid}
Different from the opposite-sign case, we will show in Theorem \ref{th:comp} that if $S_1 \neq \emptyset$ the residual $r_1$ of the first component $u_{1,k}$ in \eqref{limit} vanishes, i.e. $r_1\equiv 0$. We have been informed that this fact was already known by Prof. T. Ricciardi and Prof. T. Suzuki. This crucial property will be used in the proof of the compactness result of the Theorem \ref{th:comp}.
\end{rem}

\begin{rem} \label{rem:m-t}
By means of the local blow-up mass \eqref{mass1}, \eqref{mass2} in Theorem \ref{th:blow-up} we can interpret the sharp Moser-Trudinger inequality associated to \eqref{eq} obtained in \cite{ors}, see the discussion in Section \ref{sec:var}.
\end{rem}

\begin{rem} \label{rem:constr}
In a forthcoming paper, we shall consider the existence of blowing-up solutions to \eqref{eq} with $(m_1,m_2)$ type local mass, see \eqref{mass1}, \eqref{mass2}. On the other hand, we believe the last possibility in \eqref{mass2} could be ruled out.
\end{rem}

\medskip

The strategy of the proof for Theorem \ref{th:blow-up} is the following: by means of a selection process we can find a finite number of disks where, after scaling, the blowing-up sequence of solutions to \eqref{eq} tends to some globally defined Liouville-type equations, see Remark \ref{rem:blow-up}. In each disk the local mass is then given by  the classification results on the energy limits of the globally defined Liouville equations and the Pohozaev identity. Finally, we study the combination of the bubbling disks and get the desired result.

By exploiting the above result we will derive the following compactness property, which will be used later on.
\begin{theorem} \label{th:comp}
Let $u \in \mathring H^1(M)$ be a solution to \eqref{eq}. Suppose that $a\in\left(0,\frac 12\right)$ and that $\rho_1\notin 8\pi\N$, $\rho_2<\dfrac{8\pi}{a^2}-\dfrac{16\pi}{a}$. Then, there exists a fixed constant $C>0$ such that
$$
	|u(x)|\leq C, \qquad \forall x \in M.
$$
\end{theorem}

Notice that the assumption on $\rho_2$ is related to the minimum local blow-up mass of $(m_1,m_2)$ in \eqref{mass2}.

\medskip

In the second part of the paper we will introduce a variational argument to get the existence result in a non-coercive regime, see Theorem \ref{th:exist}. The argument will rely on the analysis developed for the mean field equation \eqref{eq:liouv}: a similar approach was used in \cite{mal-ndi,zhou} in treating the $SU(3)$ Toda system and the sinh-Gordon equation \eqref{eq:sinh}, respectively. The Euler-Lagrange functional of \eqref{eq:liouv} is given by $I_\rho : H^1(M)\to\R$,
\begin{equation} \label{liouv:func}
	I_\rho(u) = \frac{1}{2}\int_M |\nabla u|^2 \,dV_g - \rho  \left( \log\int_M h\,  e^u \,dV_g  - \int_M u \,dV_g \right).
\end{equation}
By the standard Moser-Trudinger inequality
\begin{equation}\label{ineq}
	8\pi \log\int_M e^{u-\ov u} \, dV_g \leq \frac 12 \int_M |\n u|^2\,dV_g + C_{M,g}, \qquad \ov u= \fint_M u\,dV_g,
\end{equation}
we have $I_\rho$ is bounded from below and coercive if $\rho<8\pi$. Then we can apply the direct minimization methods. For larger values of the parameter the problem becomes subtler and one needs the improved version of  the Moser-Trudinger inequality obtained in \cite{ch-li} which are based on some macroscopic division of $e^u$  over the surface. It is a key point to show that if $\rho<8(k+1)\pi, k\in\N$ and $I_\rho(u)$ is sufficient negative, then $e^u$ can not be spreading around and as a result it starts to accumulate around at most $k$ points. Such configurations are resembled by the $k$-th \emph{formal barycenters} of $M$
\begin{equation}	\label{M_k}
	M_k = \left\{ \sum_{i=1}^k t_i \d_{x_i} \, : \, \sum_{i=1}^k t_i=1, \,  x_i\in M \right\}.
\end{equation}
It is then possible to construct suitable min-max scheme based on the latter set, exploiting the crucial fact that it has non-trivial topology, and finally we can get existence of solutions to \eqref{eq:liouv}.

For the general equation \eqref{eq} there is an analogous associated functional $J_\rho : H^1(M)\to\R$, $\rho=(\rho_1,\rho_2)$
\begin{align} \label{func}
\begin{split}
	J_\rho(u) &= \frac{1}{2}\int_M |\nabla u|^2 \,dV_g - \rho_1  \left( \log\int_M h_1  e^u \,dV_g  - \int_M u \,dV_g \right) \\
	&-  {\rho_2}  \left( \log\int_M h_2  e^{a u} \,dV_g  - \int_M  a u \,dV_g \right).
\end{split}	
\end{align}
In this framework a sharp Moser-Trudinger inequality was obtained in \cite{ors} (actually, in a much more general setting). We point out that the case $supp \,\mathcal P \subset [0,1]$ presents striking differences from the opposite-sign case, see for example equation \eqref{eq:sinh}. We postpone this discussion to Section \ref{sec:var}: we suspect that the difference is due to the third possibility of the local blow-up mass \eqref{mass1}, \eqref{mass2} in Theorem \ref{th:blow-up}.

The literature on the existence issue is limited in this framework: it turns out that the sinh-Gordon case is more treatable and one can get both existence and the multiplicity results, see \cite{bjmr, jev, jev2, jev3, jwy2}. In the asymmetric case there are very few results: in the recent paper \cite{rtzz} the authors provide existence results to \eqref{eq} under the assumption that $\rho_i$ are not too large and the parameter $a$ is sufficiently small. This work is actually motivated by the latter paper and our aim is to give both a possible sharper condition on $a$ as well as some interpretation for the condition on $a$ based on the local blow-up mass obtained from Theorem \ref{th:blow-up}.

Exploiting the above argument for the mean field equation \eqref{eq:liouv} and the compactness property in Theorem~\ref{th:comp}, we can get the following existence result.
\begin{theorem} \label{th:exist}
Suppose that $a\in\left(0,\frac 12\right)$ and that $\rho_1\notin 8\pi\N$, $\rho_2<\dfrac{8\pi}{a^2}-\dfrac{16\pi}{a}$. Then, equation \eqref{eq} is solvable.
\end{theorem}

\medskip

The assumptions in the latter result are somehow sharp in the sense that they are related to the minimum local blow-up mass of $(m_1,m_2)$ given in \eqref{mass2} of Theorem~\ref{th:blow-up}.

We organize the paper in the following way: in Section \ref{sec:blow-up} we study the blow-up phenomenon related to
equation \eqref{eq} in bounded domains, we obtain the possible values of the local blow-up masses and then give the proof of the Theorem \ref{th:blow-up} and Theorem \ref{th:comp}, in Section \ref{sec:var} we introduce the min-max scheme and obtain the existence result of Theorem \ref{th:exist}.

\

\begin{center}
\textbf{Notation}
\end{center}

\medskip

The average of $u$ is denoted by $\ov u = \fint_M u \,dV_g.$ The sublevels of the functional $J_\rho$ will be denoted by
\begin{equation} \label{sub}
	J_\rho^L = \bigr\{ u\in H^1(M) \,:\, J_\rho(u) \leq L  \bigr\}.
\end{equation}
Letting $\mathcal M(M)$ be the  set of all Radon measures on $M$, the Kantorovich-Rubinstein distance is defined as
\begin{equation} \label{dist}
	\dkr (\mu_1,\mu_2) = \sup_{\|f \|_{Lip}\leq 1} \left| \int_M f \, d\mu_1 - \int_M f\,d\mu_2 \right|, \qquad \mu_1,\mu_2 \in \mathcal M(M).
\end{equation}
The symbol $B_r(p)$ stands for the open metric ball of
radius $r$ and center at $p$. When there is no ambiguity we will write $B_r\subset\R^2$ for balls which are centered at $0$.

Throughout the paper the letter $C$ will stand for positive constants which
are allowed to vary among different formulas and even within the same lines.
To stress the dependence of the constants on some
parameter  we add subscripts to $C$, for example $C_\d$. We will write $o_{\alpha}(1)$ to denote
quantities that tend to $0$ as $\alpha \to 0$ or $\alpha \to
+\infty$; the symbol
$O_\alpha(1)$ will be used for bounded quantities.

\

\section{Blow-up analysis} \label{sec:blow-up}
In this section we are going to perform the blow-up analysis and prove Theorems~\ref{th:blow-up}, \ref{th:comp}. The main point is to determine the local masses in \eqref{mass1}, \eqref{mass2}: to this end we shall restrict our attention to a blow-up sequence of solutions to \eqref{eq}. Then it is standard to get the desired conclusion, see for example Section 5 in \cite{lin-zh}. On the other hand, the convergence in \eqref{limit} is by now well-known and is obtained by similar arguments as in \cite{ri-ze}.

\medskip

Let $u_k$ be a sequence of blow-up solutions to the following equation with $(\rho_{1,k},\rho_{2,k})\to(\rho_1,\rho_2)$:
\begin{equation}
\label{seq}
	-\D u_{k} =  \rho_{1,k} h_1e^{u_{1,k}} + a  \rho_{2,k} h_2 e^{u_{2,k}} \qquad \mbox{in } B_1,
\end{equation}
where $u_{1,k}, u_{2,k}$ are defined in \eqref{norm}, such that $\int_M u_k \,dV_g = 0$ and $0$ is the only blow-up point in $B_1$, more precisely:
\begin{equation}
\label{a1}
\max_{K\subset\subset B_1\setminus\{0\}} u_{i,k}\leq C(K),\quad \max_{x\in B_1,\, i=1,2}\{u_{i,k}(x)\}\rightarrow\infty.
\end{equation}
Furthermore, we suppose that
\begin{equation}
\label{a2}
h_1(0)=h_2(0)=1, ~ \frac1C\leq h_i(x)\leq C, ~ \|h_i(x)\|_{C^3(B_1)}\leq C, \qquad \forall x\in B_1, ~i=1,2,
\end{equation}
for some constant $C>0$. We can assume
\begin{align}
\label{a3}
\begin{split}
&|u_{i,k}(x)-u_{i,k}(y)|\leq C,\qquad \forall~x,y\in\partial B_1, \qquad \int_{B_1} \rho_{i,k}h_i e^{u_{i,k}}\leq C,
\end{split}
\end{align}
where $C$ is independent of $k.$ Indeed, the assumption on the bounded oscillation can be obtained through the Green's representation. To simplify the notation, the local masses are usually defined as
\begin{align}
\label{localmass}
\sigma_i=\lim_{\delta\rightarrow0}\lim_{k\rightarrow\infty}\frac{1}{2\pi}\int_{B_{\delta}}\rho_{i,k} h_i e^{u_{i,k}}.
\end{align}

The result concerning the local masses in Theorem \ref{th:blow-up} can be rephrased as follows:
\begin{theorem}
\label{th:quantization}
Let $\sigma_i$ be defined as in (\ref{localmass}). Suppose $u_k$ satisfies (\ref{seq}), (\ref{a1}), (\ref{a3}), with $a\in(0,1)$ and $h_i$ satisfy (\ref{a2}). Then, it holds:
\begin{itemize}
\item[(a)] If $a\in\left[\frac 12,1\right)$, then $(\s_1,\s_2)$ is one of the following types:
$$
	(4,0), \quad \left(0,\dfrac{4}{a^2}\right), \quad (\a,\b),
$$
with $\a\in(0,4),\, \b\in\left(0, \dfrac{4}{a^2} \right)$ such that 
$$ \a+a\b > 4 \quad \mbox{and} \quad (\a+a \b)^2=4(\a+\b).$$
\

\item[(b)] If $a\in\left(0,\frac 12\right)$, then $(\s_1,\s_2)$ is one of the following types:
$$
	(4,0), \quad \left(0,\dfrac{4}{a^2}\right), \quad \left(4,\dfrac{4}{a^2}-\dfrac{8}{a}\right), \quad (\a,\b), \quad (4m,\g_m),
$$
with $\a\in(0,4),\, \b\in\left(\dfrac{4}{a^2}-\dfrac{8}{a}, \dfrac{4}{a^2} \right)$ such that 
$$ \a+a\b > \dfrac{4}{a}-4 \quad \mbox{and} \quad (\a+a \b)^2=4(\a+\b)$$ 
and with $m\in\N$, $m>1$, $\g_m$ such that $(4m+a \g_m)^2=4(4m+\g_m)$.
\end{itemize}
\end{theorem}

\medskip

In order to prove the Theorem \ref{th:quantization}, we introduce a suitable selection process for the equation \eqref{seq}.
\begin{proposition}
\label{pr2.1}
Let $u_k$ be a sequence of blow-up solutions to \eqref{seq} with $u_{1,k},u_{2,k}$ defined in \eqref{norm}. Suppose $u_{i,k}$ satisfies \eqref{a1} and \eqref{a3} with $a\in(0,1)$ and $h_i$ satisfies \eqref{a2}, $i=1,2$. Set $M_k(x)=\max_{i=1,2}\{ u_{i,k}(x) \}$. Then, there exist finite sequences of points $\Sigma_k:=\{x_1^k,\cdots,x_m^k\}$ (all $x_j^k\rightarrow0,~j=1,\cdots,m$) and positive numbers $l_1^k,\cdots,l_m^k\rightarrow0$ such that, 
\begin{enumerate}
  \item $M_{k,j}=\max_{i=1,2}\{u_{i,k}(x_j^k)\}=\max_{B_{l_j^k}(x_j^k), i=1,2}\{u_{i,k}\}$ for $j=1,\cdots,m.$
  \item $\exp(\frac12 M_{k,j})l_j^k\rightarrow\infty$ for $j=1,\cdots,m.$
  \item Let $\varepsilon_{k,j}=e^{-\frac12 M_{k,j}}.$ In each $B_{l_j^k}(x_j^k)$ we define the dilated function
  \begin{equation}
  \label{2.1}
  v_{i,k}=u_{i,k}(\varepsilon_{k,j}y+x_j^k)+2\log\varepsilon_{k,j}, \qquad i=1,2.
  \end{equation}
  Then we have the following possibilities:
  \begin{itemize}
    \item[(a)] $v_{1,k}$ converges to a solution of $\Delta v+\rho_1e^v=0$ and $v_{2,k}\rightarrow-\infty$ over compact subsets of $\mathbb{R}^2$, where $\rho_1=\lim_{k\rightarrow+\infty}\rho_{1,k}$  or
    \item[(b)] $v_{2,k}$ converges to a solution of $\Delta v+a^2\rho_2e^v=0$ and $v_{1,k}\rightarrow-\infty$ over compact subsets of $\mathbb{R}^2$, where $\rho_2=\lim_{k\rightarrow+\infty}\rho_{2,k}$  or
    \item[(c)] $v_{1,k},v_{2,k}$ converges to $v_1$ and $v_2$ respectively, where $v_1,v_2$ satisfies $\Delta v_1+\rho_1e^{v_1}+a\rho_2e^{v_2}=0$ and $v_2=av_1+c$ with some constant $c$.
  \end{itemize}
  \item There exits a constant $C_1>0$ independent of $k$ such that
  $$M_k(x)+2\log\mathrm{dist}(x,\Sigma_k)\leq C_1,~\forall x\in B_1.$$
\end{enumerate}
\end{proposition}

\begin{proof}
The proof is in the same spirit of the one used in \cite[Proposition 2.1]{lin-wei-zh} and \cite[Proposition 2.1]{jwy}. However, compared with the sinh-Gordon case, it is more complicate for the same sign case. Therefore, we would like to give a sketch of the proof and point out the differences.

Let $M_k(x_1^k)=\max_{x\in B_1}M_k(x).$ By assumption we clearly have $x_1^k\rightarrow0.$ Let $v_{1,k},v_{2,k}$ be defined in (\ref{2.1}) with $x_j^k,M_{k,j}$ replaced by $x_1^k,M_k(x_1^k)$ respectively. By the definition of $\varepsilon_{k,1}$ and \eqref{2.1}, we have $v_{i,k}\leq0, i=1,2.$ We note that $v_{i,k}$ satisfies,
\begin{equation}
\label{2.2}
\Delta v_{1,k}+\rho_{1,k}h_1e^{v_{1,k}}+a\rho_{2,k}h_2e^{v_{2,k}}=0, \quad \Delta v_{2,k}+a\rho_{1,k}h_1e^{v_{1,k}}+a^2\rho_{2,k}h_2e^{v_{2,k}}=0.
\end{equation}
Therefore, we can deduce that $|\Delta v_{i,k}|$ is bounded. By standard elliptic estimate, $|v_{i,k}(z)-v_{i,k}(0)|$ is uniformly bounded in any compact subset of $\R^2.$ Since $u_{1,k}$ and $u_{2,k}$ has the following relation
$$au_{1,k}=u_{2,k}+\log\int_M h_2e^{au_k}-a\log\int_M h_1e^{u_k},$$
$u_{1,k},u_{2,k}$ reaches their maximal value at the same point. If $u_{2,k}(x_1^k)\ll M_k(x_1^k)$. Then $u_{1,k}(x_1^k)=M_k(x_1^k)$, $v_{1,k}(0)=0$ and $v_{1,k}$ converges in $C_{loc}^2(\R^2)$ to a function $v_1,$ while the other component $v_{2,k}\rightarrow-\infty$ over compact subsets of $\R^2.$ As a consequence, we have the limit of $v_{1,k}$ satisfies the following equation:
\begin{equation}
\label{2.3}
\Delta v_1+\rho_1 e^{v_1}=0 \qquad \mathrm{in}~\R^2.
\end{equation}
This is one of the alternatives listed in the third conclusion and we can follow the arguments in \cite[Proposition 2.1]{jwy} to find $l_1^k$ such that
\begin{equation}
\label{2.4}
M_k(x)+2\log|x-x_1^k|\leq C,~|x-x_1^k|\leq l_1^k, \qquad e^{\frac12u_{1,k}(x_1^k)}l_1^k\rightarrow\infty.
\end{equation}
The case $u_{1,k}(x_1^k)\ll M_k(x_1^k)$ can be treated similarly and we can also find $l_1^k$ such that similar estimates as in (\ref{2.4}) hold. The left possibility is the situation when $u_{1,k}(x_1^k)$ and $u_{2,k}(x_1^k)$ are comparable. Without loss of generality, suppose $c_k=u_{2,k}(x_1^k)-u_{1,k}(x_1^k)$. According to our assumption, $c_k$ is uniformly bounded and we assume $c_k\rightarrow c_0$ as $k\rightarrow\infty$ by passing to a subsequence if necessary. Then the limit of the corresponding sequence $v_{1,k},v_{2,k}$ satisfy
\begin{equation}
\label{2.5}
\Delta v_1+\rho_1 e^{v_1}+\rho_2ae^{v_2}=0, \qquad v_2=av_1+c_0.
\end{equation}
We can rewrite the equation (\ref{2.5}) as
\begin{equation}
\label{2.6}
\Delta v_1+\rho_1e^{v_1}+\rho_2a_1e^{av_1}=0,~a_1=ae^{c_0}.
\end{equation}
By \cite[Proposition 3.2]{tar2} and \cite[Theorem 3.1]{pt}, we have equation (\ref{2.6}) admits a solution if and only if
\begin{equation}
\label{2.7}
\eta:=\frac{1}{2\pi}\int_{\R^2}\left(\rho_1e^{v_1}+\rho_2a_1e^{av_1}\right)\in\left(\max\left\{4,\frac{4}{a}-4\right\},\frac{4}{a}\right).
\end{equation}
Furthermore, we have the following estimate on the asymptotic behavior on $v_i,$
\begin{equation}
\label{2.8}
|v_1+\eta\log(|y|+1)|\leq C~\mathrm{in}~\R^2,\qquad|v_2+a\eta\log(|y|+1)|\leq C~\mathrm{in}~\R^2.
\end{equation}
From (\ref{2.7}) we can get $a\eta>2$. Then we can take $R_k\rightarrow\infty$ such that
\begin{equation}
\label{2.9}
v_{i,k}(y)+2\log|y|\leq C,~|y|\leq R_k,~i=1,2.
\end{equation}
In other words, we can find $l_1^k\rightarrow0$ such that
\begin{equation*}
M_k(x)+2\log|x-x_1^k|\leq C,~|x-x_1^k|\leq l_1^k,
\end{equation*}
and
\begin{equation*}
e^{\frac12M_k(x_1^k)}l_1^k\rightarrow\infty,~\mathrm{as}~k\rightarrow\infty.
\end{equation*}
Then we consider the function $M_k(x)+2\log|x-x_1^k|.$ If the function $M_k(x)+2\log|x-x_1^k|$ is bounded in whole $B_1$, then the selection process is terminated. Otherwise, we can use a similar argument in \cite[Proposition~2.1]{jwy} or \cite[Proposition 2.1]{lin-wei-zh} to find $x_2^k$ and $l_2^k$, where
\begin{equation*}
M_k(x)+2\log|x-x_2^k|\leq C,~|x-x_2^k|\leq l_2^k \quad \mathrm{and} \quad e^{\frac12M_k(x_2^k)}l_2^k\rightarrow\infty.
\end{equation*}
We can continue such process if the function $M_k(x)+2\log\mathrm{dist}(x,\{x_1^k,x_2^k\})$ is unbounded. Since each bubbling area contributes a positive energy, the process stops after finite steps due to \eqref{a3}. Finally we get
$$\Sigma_k=\{x_1^k,x_2^k,\cdots,x_m^k\}$$
and it holds
\begin{equation}
\label{2.10}
M_k(x)+2\log\mathrm{dist}(x,\Sigma_k)\leq C,
\end{equation}
which concludes the proof.
\end{proof}

\medskip

From the last conclusion of Proposition \ref{pr2.1}, we can get a control on the upper bound of the behavior for the blow-up solutions outside the bubbling disks and the following Harnack type inequality.
\begin{proposition}
\label{pr2.2}
For all $x\in B_1\setminus\Sigma_k,$ there exists a constant $C$ independent of $x$ and $k$ such that
\begin{equation}
\label{2.11}
|u_{i,k}(x_1)-u_{i,k}(x_2)|\leq C,~\forall x_1,x_2\in B(x,d(x,\Sigma_k)/2).
\end{equation}
\end{proposition}
\begin{proof}
We can modify the arguments in \cite[Lemma2.1]{jwy} or \cite[Lemma2.4]{lin-wei-zh} to get Proposition \ref{pr2.2}.
\end{proof}

\medskip

Let $x_k\in\Sigma_k$ and $\tau_k=\frac12d({x_k,\Sigma_k\setminus\{x_k\}})$, then for $x,y\in B_{\tau_k}(x_k)$ and $|x-x_k|=|y-x_k|$ we can get $u_{i,k}(x)=u_{i,k}(y)+O(1),~i=1,2$ from Proposition \ref{pr2.2}. Hence $u_{i,k}(x)=\overline{u}_{i,x_k}(r)+O(1)$ where $r=|x_k-x|$ and
$$\overline{u}_{i,x_k}(r)=\frac{1}{2\pi}\int_{\partial B_r(x_k)}u_{i,k}.$$
We say  $u_{i,k}$ has fast decay at $x\in B_1$ if
$$u_{i,k}(x)+2\log\mathrm{dist}(x,\Sigma_k)\leq-N_k$$
holds for some $N_k\rightarrow\infty$ for $i=1,2.$ On the other hand, we say $u_{i,k}$ has slow decay at $x$ if
$$u_{i,k}(x)+2\log\mathrm{dist}(x,\Sigma_k)\geq C,$$
for some $C>0$ independent of $k$. It is known from the following lemma that it is possible to choose $r$ such that both $u_{i,k},~i=1,2$ have the fast decay property.
\begin{lem}
\label{le2.1}
For all $\e_k\rightarrow0$ with $\Sigma_k\subset B_{\e_k/2}(0),$ there exists $l_k\rightarrow0$ such that $l_k\geq 2\e_k$ and
\begin{equation*}
\max\{\bar u_{1,k}(l_k),\bar u_{2,k}(l_k)\}+2\log l_k\rightarrow-\infty,
\end{equation*}
where $\bar u_{i,k}(r):=\frac{1}{2\pi r}\int_{\p B_r} u_{i,k}.$
\end{lem}

\medskip

Furthermore, in each bubbling disk $B_{l_j^k}(x_j^k)$ obtained in Proposition \ref{pr2.1}, we can choose some suitable ball such that both $u_{i,k},i=1,2$ have the fast decay property on the boundary of such ball: this fact plays an important role in getting the local Pohozaev identity, which is an essential tool in determining the local mass, see the Remark \ref{re2.1}.

By straightforward computations we have the following Pohozaev identity:
\begin{align}
\label{2.12}
\begin{split}
&\int_{B_r}\left(x\cdot\nabla h_1\rho_{1,k}e^{u_{1,k}}+x\cdot\nabla h_2\rho_{2,k}e^{u_{2,k}}\right)+\int_{B_r}\left(2\rho_{1,k}h_1e^{u_{1,k}}+2\rho_{2,k}h_2e^{u_{2,k}}\right)\\
&=\int_{\partial B_r}r\left(|\partial_{\nu}u_{1,k}|^2-\frac12|\nabla u_{1,k}|^2\right)+\int_{\partial B_r}r\left(\rho_{1,k}h_1e^{u_{1,k}}+\rho_{2,k}h_2e^{u_{2,k}}\right).
\end{split}
\end{align}
It is possible to choose suitable $r=l_k\rightarrow 0$ such that
$$\frac{1}{2\pi}\int_{B_{l_k}}\rho_{i,k}h_ie^{u_{i,k}}=\sigma_i+o(1),i=1,2$$
and both $u_{i,k}$ have fast decay property on $\partial B_{l_k},$ where $\sigma_i$ are introduced in \eqref{localmass}. We point out that the fast decay property is important because it leads to the second term on the right hand side of \eqref{2.12} is $o(1)$. By \eqref{2.12} one can derive that
\begin{equation}
\label{2.13}
4\left(\sigma_1+\sigma_2\right)=(\sigma_1+a\sigma_2)^2.
\end{equation}
For the detailed proof of (\ref{2.13}) we refer the readers to \cite[Proposition 3.1]{jwy}. In addition, we have the following remark which will be used frequently in the forthcoming argument.
\begin{remark}
\label{re2.1}
We have already observed that the fast decay property is crucial in evaluating the Pohozaev identity \eqref{2.12}. Moreover, let $\Sigma_k'\subseteq\Sigma_k$ and assume that
$$\mathrm{dist}\bigr(\Sigma_k',\partial B_{l_k}(p_k)\bigr)=o(1)\,\mathrm{dist}\bigr(\Sigma_k\setminus\Sigma_k',\partial B_{l_k}(p_k)\bigr).$$
If both components $u_{i,k},i=1,2$ have the fast decay property on $\partial B_{l_k}(p_k)$, namely
$$
	\max\{u_{i,k}(x)\} \leq -2\log|x-p_k| -N_k, \quad x\in \partial B_{l_k}(p_k),
$$
for some $N_k\to+\infty$. Then, we can evaluate a local Pohozaev identity as in \eqref{2.12} and get
\begin{align*}
\bigr(\tilde{\sigma}_1^k(l_k)+a\tilde{\sigma}_2^k(l_k)\bigr)^2
=4\left(\tilde{\sigma}_1^k(l_k)+\tilde{\sigma}_2^k(l_k)\right)+o(1),
\end{align*}
where
\begin{align*}
\tilde{\sigma}_i^k(l_k) = \frac{1}{2\pi}\int_{B_{l_k}(p_k)}\rho_{i,k}h_i e^{u_{i,k}}, \qquad i=1,2.
\end{align*}
We note that if $B_{l_k}(p_k) \cap \Sigma_k=\emptyset$ then $\tilde{\sigma}_i^k(l_k)=o(1), i=1,2$ and the above formula clearly holds.
\end{remark}

\medskip

To understand the local energy we have to study the contribution of the energy from each bubbling disk. Without loss of generality we may assume that $0\in\Sigma_k$ (after suitable translation if necessary). Let $\tau_k=\frac12\mathrm{dist}(0,\Sigma_k\setminus\{0\})$. We set
$$\sigma_i^k(r)=\frac{1}{2\pi}\int_{B_r(0)}\rho_{i,k}h_ie^{u_{i,k}}, i=1,2,$$
for $0<r\leq\tau_k$ and $\overline{u}_{i,k}(r)=\frac{1}{2\pi r}\int_{\partial B_r(0)}u_{i,k}.$ By using equation \eqref{seq} we get the following key property:
\begin{align}
\label{deriv}
\begin{split}
&\frac{d}{dr}\overline{u}_{1,k}(r)=\frac{1}{2\pi r} \int_{\p B_r} \frac{\p u_{1,k}}{\p \nu} = \frac{1}{2\pi r} \int_{B_r} \D u_{1,k} = \frac{-\sigma_1^k(r)-a\sigma_2^k(r)}{r},\\
&\frac{d}{dr}\overline{u}_{2,k}(r)=\frac{1}{2\pi r} \int_{\p B_r} \frac{\p u_{2,k}}{\p \nu} = \frac{1}{2\pi r} \int_{B_r} \D u_{2,k} = \frac{-a\sigma_1^k(r)-a^2\sigma_2^k(r)}{r}.
\end{split}
\end{align}
Moreover, from the selection process we have
\begin{equation*}
\max\{u_{1,k}(x),u_{2,k}(x)\}+2\log|x|\leq C, \qquad |x|\leq\tau_k.
\end{equation*}
We recall that if both components have fast decay property on $\p B_r(0)$ for $r\in(0,\tau_k)$, then $\sigma_1^k(r),\sigma_2^k(r)$ satisfy
\begin{align*}
\left(\sigma_1^k(r)+a\sigma_2^k(r)\right)^2=4\left(\sigma_1^k(r)+\sigma^k_2(r)\right) +o(1),
\end{align*}
see Remark \ref{re2.1}. Furthermore, we have the following result on the description of the energy contributed in $B_{\tau_k}(0).$
\begin{proposition}
\label{pr2.3}
Suppose (\ref{seq})-(\ref{a3}) hold for $u_k$ and $h_i$. Then, we have:
\begin{itemize}
\item[(a)] If $a\in\left[\frac 12,1\right)$, then $\left(\s_1^k(\tau_k),\s_2^k(\tau_k)\right)$ is a small perturbation of one of the following types:
$$
	(4,0), \quad \left(0,\dfrac{4}{a^2}\right), \quad (\a,\b),
$$
with $\a, \b$ as in the point \emph{(a)} of the Theorem \ref{th:quantization} and both $u_{i,k}$ are fast decay on $\p B_{\tau_k}$.

\

\item[(b)] If $a\in\left(0,\frac 12\right)$, then $\left(\s_1^k(\tau_k),\s_2^k(\tau_k)\right)$ is either a small perturbation of one of the following types:
$$
	(4,0), \quad \left(0,\dfrac{4}{a^2}\right), \quad \left(4,\dfrac{4}{a^2}-\dfrac{8}{a}\right), \quad (\a,\b),
$$
with $\a, \b$ as in the point \emph{(b)} of the Theorem \ref{th:quantization} and both $u_{i,k}$ are fast decay on $\p B_{\tau_k}$, or $\sigma_1^k(\t_k)=4+o(1)$ and $u_{1,k}$ is fast decay on $\p B_{\tau_k}$.
\end{itemize}
\end{proposition}

\medskip

\begin{rem}
The latter result is different in nature from the sinh-Gordon \eqref{eq:sinh} and other opposite-sign cases where the fast decaying component is the one with the bigger local mass. On the contrary, here we can assert both components have fast decay if we exclude the $\sigma_1^k(\t_k)=4+o(1)$ case.
\end{rem}

\begin{proof}[Proof of Proposition \ref{pr2.3}]
The proof is mainly followed by the argument in \cite[Proposition 5.1]{lin-wei-zh} and \cite[Proposition 4.1]{jwy}, with some modifications. As explained in the proof of Proposition~\ref{pr2.1}, both components obtain its maximal at the same point, which is here by assumption $x_1^k=0.$ Let
$$
-2\log\delta_k=\max\{u_{1,k}(x_1^k),u_{2,k}(x_1^k)\}=\max_{x\in B_{\tau_k}(0)}\max\{u_{1,k}(x),u_{2,k}(x)\}.
$$
We set
\begin{equation*}
\label{eqv}
v_i^k(y)=u_{i,k}(\delta_ky)+2\log\delta_k, \qquad |y|\leq\frac{\tau_k}{\delta_k}.
\end{equation*}
If the maximal values of $u_{i,k}(x)$ are comparable: as in Proposition \ref{pr2.1} it holds that $v_i^k$ converges to $v_i$, where $v_i$ verifies the following
\begin{equation}
\label{eqlimit}
\Delta v_1+\rho_1e^{v_1}+a\rho_2e^{v_2}=0, \qquad v_2=av_1+c_0,
\end{equation}
for some constant $c_0$. Then, from \cite[Proposition 3.2]{tar2} and \cite[Theorem 3.1]{pt} we have \eqref{eqlimit} admits a solution if and only if
\begin{equation}
\label{2.14}
\eta=\frac{1}{2\pi}\int_{\R^2}\left(\rho_1e^{v_1}+a\rho_2e^{v_2}\right) \in\left(\max\left\{4,\frac{4}{a}-4\right\},\frac{4}{a}\right).
\end{equation}
In addition, we have the following estimate on the asymptotic behavior on $v_i,$
\begin{align}
\label{2.15}
|v_1+\eta\log(|y|+1)|\leq C~\mathrm{in}~\R^2,\qquad|v_2+a\eta\log(|y|+1)|\leq C~\mathrm{in}~\R^2.
\end{align}
From (\ref{2.14}) we have $a\eta>2$. As a consequence, we can choose $R_k\rightarrow\infty$ (we assume that $R_k=o(1)\tau_k/\delta_k$) such that
\begin{equation}
\label{2.16}
\frac{1}{2\pi}\int_{B_{R_k}}\rho_{i,k}h_i(\delta_ky)e^{v_i^k}=\frac{1}{2\pi}\int_{\R^2}\rho_ie^{v_i}+o(1).
\end{equation}
For $r\geq R_k$ we point out that
\begin{equation*}
\sigma_i^k(\delta_kr)=\frac{1}{2\pi}\int_{B_r}\rho_{i,k}h_i(\delta_ky)\,e^{v_i^k}.
\end{equation*}
Then we get $\sigma_i^k(\delta_kR_k)=\frac{1}{2\pi}\int_{\R^2}\rho_ie^{v_i}+o(1)$. Now, we consider the behavior of the solutions from $B_{\delta_kR_k}$ to $B_{\tau_k}$. First we note that on $\p B_{\delta_kR_k}$ by estimate of the local energies in \eqref{2.14} and by \eqref{deriv} we get
\begin{align*}
\frac{d}{dr}(\bar u_{i,k}(r)+2\log r)<0.
\end{align*}
The latter property together with (\ref{2.16}) implies that both $u_{i,k}$ are fast decay from $B_{\delta_kR_k}$ to $B_{\tau_k}$. Therefore, following a similar argument in \cite[Lemma 4.1]{jwy}, we conclude that the energy of each component only changes by o(1) and both components remains fast decay on $\p B_{\tau_k}.$ This gives raise to the $(\a,\b)$ type in the points (a), (b).

\medskip

If the maximal values of $u_{i,k}$ are not comparable, then we have two cases: $u_{1,k}(x_1^k)\ll u_{2,k}(x_1^k)$ or $u_{1,k}(x_1^k)\gg u_{2,k}(x_1^k)$. When the former case happens, we shall get $v_{1,k}\rightarrow-\infty$ uniformly in any compact subsets of $\R^2$, while $v_{2,k}$ converges to a solution which satisfies
\begin{equation}
\label{2.17}
\Delta v+a^2\rho_2e^{v}=0.
\end{equation}
Then by the quantization of the limit equation (\ref{2.17}), we can choose $R_k\rightarrow\infty$ such that
\begin{equation}
\label{2.18}
\frac{1}{2\pi}\int_{B_{R_k}}\rho_{1,k}h_1(\delta_ky)e^{v_1^k(y)}=o(1),
 \quad \frac{1}{2\pi}\int_{B_{R_k}}\rho_{2,k}h_2(\delta_ky)e^{v_2^k(y)}=\frac{4}{a^2}+o(1).
\end{equation}
For $r\geq R_k$ we clearly have
\begin{equation*}
\sigma_i^k(\delta_kr)=\frac{1}{2\pi}\int_{B_r}\rho_{i,k}h_i(\delta_ky)e^{v_i^k(y)}.
\end{equation*}
Then we have $\s_1^k(\delta_kR_k)=o(1)$ and $\s_2^k(\delta_kR_k)=\dfrac{4}{a^2}+o(1).$ By using the equation \eqref{deriv}, we have
\begin{align*}
\frac{d}{dr}(\bar u_{i,k}(r)+2\log r)\leq 0,
\end{align*}
as before, we could get both components are fast decay up to $\p B_{\tau_k}$ and each component only changes by $o(1).$ Hence, we end up with the $\left(0,\dfrac{4}{a^2}\right)$ type in the points (a), (b).

\medskip

The left case is $u_{1,k}\gg u_{2,k}$. Similarly, we could find $R_k\rightarrow\infty$ such that
\begin{equation}
\label{2.19}
\frac{1}{2\pi}\int_{B_{R_k}}\rho_{1,k}h_1(\delta_ky)e^{v_1^k(y)}=4+o(1),
 \quad \frac{1}{2\pi}\int_{B_{R_k}}\rho_{2,k}h_2(\delta_ky)e^{v_2^k(y)}=o(1).
\end{equation}
Then we have $\s_1^k(\delta_kR_k)=4+o(1)$ and $\s_2^k(\delta_kR_k)=o(1).$ If $a\geq \frac 12$ then $\frac{d}{dr}(\bar u_{2,k}(r)+2\log r) \leq 0$ and as before each component only changes by $o(1)$: we get the $(4,0)$ type in the point (a). On the other hand, for $a<\frac12$ we can not determine the sign of $\frac{d}{dr}(\bar u_{2,k}(r)+2\log r)$. Then we get the following alternatives: for $r\geq R_k$ either both $v_{i,k}$ have fast decay property up to $\partial B_{\tau_k}$ (and we get $\left(\sigma_1^k(\t_k),\sigma_2^k(\t_k)\right)=(4,0)+o(1)$) or there exists $L_k\in (R_k,\tau_k/\delta_k)$ such that $v_{2,k}$ has the slow decay
\begin{equation*}
v_2^k(y)\geq-2\log L_k-C, \quad |y|=L_k,
\end{equation*}
for some $C>0$, while
$$v_1^k(y)\leq-2\log|y|-N_k, \qquad |R_k|\leq |y|\leq L_k,$$
for some $N_k\rightarrow+\infty$. If the latter happens and $L_k=o(1)\tau_k/\delta_k,$ we can repeat the arguments in \cite[Lemma 4.2]{jwy} to find $\tilde L_k$ such that $\tilde L_k/L_k\rightarrow\infty$ and $\tilde L_k=o(1)\tau_k/\delta_k$ and both components have fast decay:
$$v_i^k(y)\leq -\log|y|-N_k, \qquad |y|=\tilde L_k,~i=1,2,$$
for some $N_k\rightarrow\infty.$ Moreover, by the Pohozaev identity, see Remark \ref{re2.1},  $\sigma_1^k(\delta_k\tilde L_k)=4+o(1)$ and $\sigma_2^k(\delta_k\tilde L_k)=\dfrac{4}{a^2}-\dfrac{8}{a}$. Then as $r$ grows from $\delta_k\tilde L_k$ to $\tau_k$, we have both $u_{i,k}$ satisfy $\frac{d}{dr}(\bar u_{i,k}(r)+2\log r)\leq 0$. Hence, both components have fast decay up to $\partial B_{\tau_k}$ and the energies are $\sigma_i^k(\delta_k\tilde L_k)+o(1)$, respectively. Hence, we end up with the $\left(4,\dfrac{4}{a^2}-\dfrac{8}{a}\right)$ type in the point (b).

If instead $L_k=O(1)\tau_k/\delta_k$, by Proposition \ref{pr2.2} we directly conclude that the first component has fast decay while the second one has slow decay. Moreover, $\sigma_1^k(\t_k)=4+o(1)$, while $\sigma_2^k(\t_k)$ can not be determined at this point. Thus, we finish the proof.
\end{proof}

\medskip

After analyzing the behavior of the bubbling solution $u_{i,k}$ in each disk, we turn to consider the combination of the bubbling disks in a group. The concept of group for this kind of problems was first introduced in \cite{lin-wei-zh}. Roughly speaking, the groups are made of points in $\Sigma_k$ which are relatively close to each other but relatively far away from the other points in $\Sigma_k$.

\medskip

\noindent {\bf Definition.} Let $G=\{p_1^k,\cdots,p_q^k\}$ be a subset of $\Sigma_k$ with more than one point in it. $G$ is called a group if
\begin{enumerate}
  \item dist$(p_i^k,p_j^k)\sim $ dist$(p_s^k,p_t^k)$,
 where $p_i^k,p_j^k,p_s^k,p_t^k$ are any points in $G$ such that $p_i^k\neq p_j^k$ and $p_t^k\neq p_s^k.$ \vspace{0.3cm}

  \item $\dfrac{\mbox{dist}(p_i^k,p_j^k)}{\mbox{dist}(p_i^k,p_k)}\rightarrow0$,
  for any $p_k\in\Sigma_k\setminus G$ and for all $p_i^k,p_j^k\in G$ with $p_i^k\neq p_j^k.$
\end{enumerate}

\medskip

We note that from Proposition \ref{pr2.2} if both components $u_{i,k}$ have fast decay around one of the disks in a group, they are forced to have fast decay around all the disks in this group. More precisely, suppose $B_{\tau_1^k}(x_1^k),\cdots,B_{\tau_m^k}(x_m^k)$ are all the bubbling disks in some group, where $\tau_j^k=\frac12\mathrm{dist}(x_j^k,\Sigma_k\setminus\{x_j^k\})$. By the definition of group, all the $\tau_l^k,~l=1,\cdots,m$ are comparable. Suppose both $u_{i,k}$ have fast decay: then we can find $N_k\rightarrow\infty$ such that all the disks in this group are contained in $B_{N_k\tau_1^k}(0)$ and
\begin{align} \label{tot}
\sigma_i^k\left(N_k\tau_1^k\right)=\sum_{j=1}^m\sigma_i^k\left(B_{\tau_{j}^k}(x_j^k)\right)+o(1), \qquad i=1,2,
\end{align}
where
$$\sigma_i^k\left(B_{\tau_{j}^k}(x_j^k)\right)=\frac{1}{2\pi}\int_{B_{\tau_{j}^k}(x_j^k)}\rho_{i,k}h_ie^{u_{i,k}}, \qquad i=1,2.$$
Roughly speaking, the energy contribution comes just from the energy in each bubbling disk.

From Proposition \ref{pr2.3}, if there is a bubbling disk such that
\begin{align}
\label{2.20}
\begin{split}
\left(\sigma_1^k\left(B_{\tau_{j}^k}(x_j^k)\right),\sigma_2^k\left(B_{\tau_{j}^k}(x_j^k)\right)\right)
\in&\Biggr\{  \left(0,\frac{4}{a^2}\right)+o(1),\left(4,\frac{4}{a^2}-\frac{8}{a}\right)+o(1),\Biggr. \\
&\Biggr.(\alpha,\beta)+o(1)\Biggr\},
\end{split}
\end{align}
where $(\alpha,\beta)$ satisfies the relation listed in Proposition \ref{pr2.3}, we can conclude both components $u_{i,k},i=1,2$ have fast decay. Then we claim there is at most one bubbling disk in this group. If not, suppose there are more than one bubbling disk: $B_{\tau_1^k}(x_1^k),\cdots,B_{\tau_m^k}(x_m^k),~m\geq2$. From Remark \ref{re2.1} $\sigma_i^k\left(B_{\tau_{j}^k}(x_j^k)\right)$ satisfies
\begin{align}
\label{2.21}
\left(\sigma_1^k\left(B_{\tau_{j}^k}(x_j^k)\right)+a\sigma_2^k\left(B_{\tau_{j}^k}(x_j^k)\right)\right)^2=
4\left(\sigma_1^k\left(B_{\tau_{j}^k}(x_j^k)\right)+\sigma_2^k\left(B_{\tau_{j}^k}(x_j^k)\right)\right) + o(1),
\end{align}
for $j=1,\cdots,m$ and, by \eqref{tot}
\begin{align}
\label{2.22}
\begin{split}
\left(\sum_{j=1}^m\sigma_1^k\left(B_{\tau_{j}^k}(x_j^k)\right)
+a\sum_{j=1}^m\sigma_2^k\left(B_{\tau_{j}^k}(x_j^k)\right)\right)^2 =~&
4\left(\sum_{j=1}^m\sigma_1^k\left(B_{\tau_{j}^k}(x_j^k)\right)\right. \\
&\left. +\sum_{j=1}^m\sigma_2^k\left(B_{\tau_{j}^k}(x_j^k)\right)\right) + o(1).
\end{split}
\end{align}
However, when $m\geq2$, it is impossible for \eqref{2.21} and \eqref{2.22} to hold simultaneously. Thus, we prove the claim.

\medskip

Therefore, if there are more than one bubbling disk in some group, then all of the local energies $\sigma_1^k\left(B_{\tau_{j}^k}(x_j^k)\right)$ must be a small perturbation of $4$ with $u_{1,k}$ fast decay and $u_{2,k}$ slow decay. Then we can choose $N_k\rightarrow\infty$ such that both $u_{i,k}$ are fast decay on $\p B_{N_k\tau_k}(0)$ and
\begin{align}
\label{2.23}
\sigma_1^k\left(B_{N\tau_k}(0)\right)=\sum_{j=1}^m\sigma_1^k\left(B_{\tau_j^k}(x_j^k)\right)+o(1)=4m+o(1),
\end{align}
with $m>1$. Furthermore, $\sigma_1^k\left(B_{N\tau_k}(0)\right),\sigma_2^k\left(B_{N\tau_k}(0)\right)$ satisfy
\begin{align}
\label{2.24}
\left(\sigma_1^k\left(B_{N\tau_k}(0)\right)+a\sigma_2^k\left(B_{N\tau_k}(0)\right)\right)^2=
4\left(\sigma_1^k\left(B_{N\tau_k}(0)\right)+\sigma_2^k\left(B_{N\tau_k}(0)\right)\right).
\end{align}
If the quadratic equation $(4m+ax)^2=4(4m+x)$ admits solution $\gamma_{m}$, i.e.,
\begin{align}
\label{2.25}
(4m+a\gamma_m)^2=4(4m+\gamma_m),
\end{align}
then we can get
\begin{align}
\label{2.26}
\left(\sigma_1^k\left(B_{N\tau_k}(0)\right),\sigma_2^k\left(B_{N\tau_k}(0)\right)\right)=(4m,\gamma_m)+o(1).
\end{align}
Since $m>1$ we observe that for $a\geq \frac 12$ the equation \eqref{2.25} does not admit solutions and hence this type of local mass is not present for this range of the parameter $a$. Furthermore, from \eqref{deriv} we can see that $u_{1,k}$ is still the fast decay component in $B_{N\tau_k}\setminus B_{\tau_k}.$ In conclusion, if there is only one bubbling disk, then the local energy of the area covered by this group is one of the following
\begin{align*}
\left\{(4,0),\left(0,\frac{4}{a^2}\right),\left(4,\frac{4}{a^2}-\frac{8}{a}\right),(\alpha,\beta)\right\},
\end{align*}
where $(\alpha,\beta)$ satisfy the relation listed in Proposition \ref{pr2.3} (the third type is not present for $a\geq \frac 12$), or there are multiple bubbling disks in this group, and the associated local energy is $(4m,\gamma_m)$ where $\gamma_m$ satisfies $(4m+a\gamma_m)^2=4(4m+\gamma_m)$, $m>1$. We point out that except for the case $(\a,\b)$ the local energy of the first component is always a perturbation of a multiple of $4$ and $u_{1,k}$ is fast decaying. Then, we start to combine the groups, which is very similar as the combination of the bubbling disks. We continue to include the groups which are far away from $0$: from the selection process we have only finite bubbling disks and as a result the combination procedure will terminate in finite steps. Then, we can take $s_k\rightarrow0$ with $\Sigma_k\subset B_{s_k}(0)$ such that both components $u_{i,k}$ are fast decay on $\p B_{s_k}(0)$. Therefore, we have $\sigma_1^k(s_1^k),\sigma_2^k(s_2^k)$ is a small perturbation of one of the following types
\begin{align*}
\left\{(4,0), \left(0,\frac{4}{a^2}\right), \left(4,\frac{4}{a^2}-\frac{8}{a}\right), (\alpha,\beta), (4m,\gamma_m)\right\},
\end{align*}
where $\gamma_m$ satisfies (\ref{2.25}), $m>1$, and $(\alpha,\beta)$ satisfy the condition stated in Proposition~\ref{pr2.3} (the third and the last types are not present for $a\geq \frac 12$). On the other hand, we have
$$\sigma_i=\lim_{k\rightarrow\infty}\sigma_i^k(s_k),\qquad i=1,2.$$
It follows that $\sigma_1,\sigma_2$ satisfy the quantization property of Theorem \ref{th:quantization} and we finish the proof.

\

We end this section by giving the proof of the compactness property of Theorem~\ref{th:comp}.

\medskip

\noindent {\em Proof of Theorem \ref{th:comp}.} We prove this theorem by contradiction. Suppose the conclusion is wrong. Then $S\neq\emptyset.$ In order to derive a contradiction, we need an improved version of the second conclusion in Theorem \ref{th:blow-up}. We claim that if $S_1\neq\emptyset$, then
\begin{align}
\label{2.27}
\rho_{1,k}h_1e^{u_{1,k}}\rightharpoonup \sum_{p\in S_1}m_1(p)\delta_p,
\end{align}
in other words, we claim $r_1\equiv 0$ if $S_1\neq\emptyset$. It is equivalent to $\int_Mh_1e^{u_k}\rightarrow\infty$. If the claim is not true, then we have $\int_Mh_1e^{u_k}$ is uniformly bounded and so is $\int_Mh_2e^{au_k}$ by Holder's inequality. Letting
$$C_{1k}^{-1}=\int_Mh_1e^{u_k}, \qquad C_{2k}^{-1}=\int_Mh_2e^{au_k},$$
we can find two positive constants $C_1$ and $C_2$ such that $C_1\leq C_{ik}\leq C_2.$
Then, we can write \eqref{eq} as (recall $|M|=1$)
\begin{align*}
\Delta u_k+C_{1k}\rho_{1,k} h_1 e^{u_k}+aC_{2k}\rho_{2,k}h_2e^{au_k}-\rho_{1,k}-a\rho_{2,k}=0.
\end{align*}
We decompose $u_k=v_{1k}+v_{2k}$, where $v_{ik}$ satisfy
\begin{align*}
\left\{
\begin{array}{ll}
\Delta v_{1k}+C_{1k}\rho_{1,k}h_1e^{v_{2k}}e^{v_{1k}}=0 & \mathrm{in}~B_r(p), \\
v_{1k}=0 & \mathrm{on}~\partial B_r(p),
\end{array}
\right.
\end{align*}
and
\begin{align*}
\left\{
\begin{array}{ll}
\Delta v_{2k}+aC_{2k}\rho_{2,k}h_2e^{au_k}-\rho_{1,k}-a\rho_{2,k}=0 & \mathrm{in}~B_r(p),\\
v_{1k}=u_k &\mathrm{on}~\partial B_r(p),
\end{array}
\right.
\end{align*}
where $r$ is chosen such that $B_r(p)\cap S=\{p\}$. Since $\int_{B_r(p)}e^{u_k}$ is uniformly bounded we have $e^{au_k}\in L^{\frac{1}{a}}(B_r(p)).$ Using Green's representation formula we can conclude $u_k$ is uniformly bounded in $M\setminus S$. By standard elliptic estimate, we can get $v_{2k}$ is uniformly bounded in $B_r(p)$. Then $v_{1k}$ satisfies
\begin{equation}
\label{2.28}
\left\{
\begin{array}{ll}
\Delta v_{1k}+h_ke^{v_{1k}}=0 & \mathrm{in}~B_r(p),\\
v_{1k}=0 &\mathrm{on}~\partial B_r(p),
\end{array}
\right.
\end{equation}
where $h_k=C_{1k}\rho_{1,k}h_1e^{v_{2k}}$ is uniformly bounded. Since $u_k$ blows-up at $p$ we apply Brezis-Merle's result \cite{bm} to deduce $C_{1k}\rho_{1,k}h_1e^{u_k}=h_k e^{v_{1k}}\rightharpoonup 8\pi\delta_p$ in $B_r(p)$. This contradicts to the assumption $\int_Mh_1e^{u_k}$ is bounded. Hence, we finish the proof of the claim.

\medskip

It follows that if $S_1\neq\emptyset$ we get $\rho_{1,k}h_1e^{u_{1,k}}\rightharpoonup\sum_{p\in S_1}m_1(p)\delta_p$. Since we are assuming $a<\frac 12$ and $\rho_2<\dfrac{8\pi}{a^2}-\dfrac{16\pi}{a}$, we can not have a local mass of the type $(\a,\b)$ in the Theorem \ref{th:blow-up}. Therefore, the local mass of $u_{1,k}$ is a multiple of $8\pi$ around the blow-up point, i.e. $m_1(p)$ is multiple of $8\pi$ for every $p\in S_1$. Then, we get $\rho_1\in 8\pi\mathbb{N}$. Contradiction arises. Therefore, $S_1=\emptyset$ and $S_2\neq\emptyset.$ Since the blow-up mass $(8\pi,0)$ was already considered in the latter argument, from Theorem~\ref{th:blow-up} we have $m_2(p)=\dfrac{8\pi}{a^2}$ for any blow-up point $p\in S_2$. This leads to $\rho_2>\dfrac{8\pi}{a^2}$ and it contradicts the assumption $\rho_2<\dfrac{8\pi}{a^2}-\dfrac{16\pi}{a}.$ Thus, $S_2$ is also empty. Therefore, we prove the theorem.
\begin{flushright}
$\square$
\end{flushright}

\

\section{Variational analysis} \label{sec:var}

In this section we will introduce the variational argument needed for the existence result of Theorem~\ref{th:exist}. We will start by discussing about the Moser-Trudinger inequality associated to \eqref{eq} and its improved version. Next we construct a min-max scheme based on the barycenters of $M$, see \eqref{M_k}, which yields the existence of solutions to \eqref{eq}. The latter argument is quite standard now, see for example \cite{mal-ndi,zhou}, so we shall only give the crucial steps.

\subsection{Moser-Trudinger inequality}

As already pointed out in the Introduction, a sharp Moser-Trudinger inequality in this setting was obtained in \cite{ors}. More precisely, let
\begin{equation} \label{func:prob}
	J_{\mathcal P}(u) = \frac{1}{2}\int_M |\nabla u|^2 \,dV_g - \rho \int_I  \left( \log\int_M  e^{\a (u-\ov u)} \,dV_g  \right) \, \mathcal P(d\a)
\end{equation}
be the energy functional associated to the general problem \eqref{eq:prob}, $J_{\mathcal P}$ is bounded from below if and only if
\begin{equation} \label{best}
	\rho \leq 8\pi \inf \left\{ \frac{\mathcal P(K_{\pm})}{\left( \int_{K_{\pm}} \a \,\mathcal P(d\a) \right)^2} \,:\, K_{\pm}\subset I_{\pm}\cap supp\,\mathcal P \right\},
\end{equation}
where $K$ is a Borel set and $I_+=[0,1]$ and $I_-=[-1,0)$. In the case $\mathcal P(d\a)= \wtilde{\mathcal P}(d\a) = \t \d_1(d\a)+(1-\t) \d_{a}(d\a)$, with $a\in(0,1)$ and $\t_1\in[0,1]$, the functional \eqref{func:prob} is reduced to
\begin{align*}
	J_{\wtilde{\mathcal P}}(u) =~&\frac{1}{2}\int_M |\nabla u|^2 \,dV_g - \rho \t  \left( \log\int_M h_1  e^u \,dV_g  - \int_M u \,dV_g \right) \\
	&-  \rho (1-\t) \left( \log\int_M h_2  e^{a u} \,dV_g  - \int_M  a u \,dV_g \right),
\end{align*}
while the sharp inequality \eqref{best} is given by
\begin{equation} \label{best2}
	\rho \leq 8\pi\min \left\{ \frac{1}{\t}, \frac{1}{a^2(1-\t)}, \frac{1}{(\t+a(1-\t))^2} \right\},
\end{equation}
by taking $K=\{1\}, \, K=\{a\}$ and $K=\{a, 1\}$, respectively. In the notation of the functional $J_\rho$ in \eqref{func}, where $\rho=\rho_1+\rho_2$ and $\t= \frac{\rho_1}{\rho_1+\rho_2}$, $1-\t= \frac{\rho_2}{\rho_1+\rho_2}$, we get equivalently
$$
		1 \leq 8\pi\min \left\{ \frac{1}{\rho_1}, \frac{1}{a^2\rho_2}, \frac{\rho_1+\rho_2}{(\rho_1+a\rho_2)^2} \right\},
$$
which implies that all the following conditions have to be satisfied:
\begin{equation} \label{best3}
	\rho_1 \leq 8\pi, \qquad \rho_2 \leq \frac{8\pi}{a^2}, \qquad (\rho_1+a\rho_2)^2 \leq 8\pi (\rho_1+\rho_2).
\end{equation}
Some comments are needed here: the first value in \eqref{best3} is related to the standard Moser-Trudinger inequality \eqref{ineq}, while the first two values are the sharp inequality's constants in the opposite-sign case, see for example \eqref{eq:sinh}. On the other hand, in the same-sign case the sharp inequality changes due to the last value in \eqref{best3}. By means of the Theorem \ref{th:blow-up} we can interpret the latter quantity in terms of the local blow-up masses in \eqref{mass1}, \eqref{mass2}. Indeed, the first two values in \eqref{best3} correspond to the first two masses in \eqref{mass1}, \eqref{mass2} while the last one in \eqref{best3} is related to the fully blow-up mass, see also Remark \ref{rem:blow-up}, which is the third (resp. fourth) possibility in \eqref{mass1} (resp. \eqref{mass2}). The third local mass in \eqref{mass2} is nothing but the limit case of the fully blow-up situation and represents the infimum of the latter fully blow-up masses. On the other hand, the last type of local mass in \eqref{mass2}, which we denote here by $(\s_1,\s_2)$, does not give any further restriction on the parameters in the Moser-Trudinger inequality since $\s_1 > 8\pi$.

\medskip

Let us see now what kind of Moser-Trudinger inequality we will need in the sequel. Recall that in the existence result of Theorem \ref{th:exist} we assume $a<\frac 12$ and $\rho_2<\dfrac{8\pi}{a^2}-\dfrac{16\pi}{a}$. Under the latter assumptions we can derive by straightforward calculations that the sharp condition in \eqref{best3} is given by $\rho_1 \leq 8\pi$ (namely, the last bound in \eqref{best3} does not give any restriction on $\rho_1$). In other words, we deduce the following inequality:
\begin{equation} \label{m-t}
8\pi \log\int_M e^{u-\ov u} \, dV_g + \rho_{2,a} \log\int_M e^{a(u-\ov u)} \, dV_g \leq \frac 12 \int_M |\n u|^2\,dV_g + C_{M,g},
\end{equation}
where $\ov u= \fint_M u\,dV_g$ and we set $\rho_{2,a}=\dfrac{8\pi}{a^2}-\dfrac{16\pi}{a}$. We point out that our assumptions are sharp in the sense that if $\rho_2>\rho_{2,a}$, the latter inequality does not hold true with respect to $(8\pi,\rho_2)$.

Inequality \eqref{m-t} gives boundedness from below and coercivity of the functional $J_\rho$ in \eqref{func} for $\rho_1<8\pi$ and $\rho_2 < \rho_{2,a}$. In the non-coercive regime what we really need is an improved version of it. It is by now well-known how to deduce following type of results, see in particular \cite{zhou} (see also \cite{bjmr, je-ya} for more general results), hence we only sketch the process.
\begin{pro}\label{mt-impr}
Let $\rho_{2,a}$ be defined as above and let $\dis{\d>0}$, $\dis{\th>0}$, $\dis{k\in\N}$, $\dis{\{S_{i}\}_{i=1}^k\subset M}$ be such that $d(S_i,S_j)\geq \d$ for $i\neq j$. Then, for any $\dis{\e>0}$ there exists $\dis{C=C\left(\e,\d,\th,k,M\right)}$ such that if $u\in H^1(M)$ satisfies
\begin{align*}
\int_{S_{i}} e^{u}\,dV_g\ge\th\int_M e^{u}\,dV_g, \qquad \forall i\in\{1,\dots,k\},
\end{align*}
it follows that
$$8k\pi\log\int_M e^{u-\ov{u}}\,dV_g+\rho_{2,a}\log\int_M e^{a(u-\ov u)}\,dV_g\leq \frac{1+\e}{2}\int_M |\n u|^2\,dV_g+C. $$
\end{pro}

\begin{proof}
We list the main steps for the reader's convenience. One may assume $\ov u=0$ and $u$ is decomposed so that $u=v+w$, with $\ov v = \ov w = 0$ and $v\in L^{\infty}(M)$. Such decomposition will be suitably chosen in the final step. By a covering argument, see for example \cite{bjmr, je-ya, mal-ndi}, there exist $\dis{\ov\d>0,\;\ov\th>0}$ and $\dis{\{\O_n\}_{n=1}^k\subset M}$ such that $d(\O_i,\O_j)\geq \ov\d$ for $i\neq j$ and
\begin{equation} 	\label{vol}
\int_{\O_n}e^{u}\,dV_g\ge\ov\th {\int_M e^{u}\,dV_g}, \quad \int_{\O_1}e^{au}\,dV_g \ge\ov\th\int_M e^{au}\,dV_g, \qquad\forall\;n\in\{1,\dots,k\}.
\end{equation}
Next we take $k$ cut-off functions $0\leq\chi_n\leq 1$ such that
\begin{equation} \label{cut}
	{\chi_n}_{|\O_n} \equiv 1, \quad {\chi_n}_{|M\setminus (B_{\d/2}(\O_n))}\equiv 0, \quad |\n \chi_n| \leq C_{\d}, \qquad n=1,\dots,k.
\end{equation}
The assumption \eqref{vol} on the volume spreading of $u$ implies
\begin{align}
	\log \int_M e^u \, dV_g &\leq \log \int_{\O_n} e^u \,dV_g + C_{\d} \leq \log \int_{M} e^{\chi_n w} \,dV_g + \|v\|_{L^{\infty}(M)}+ C. \label{dec}
\end{align}
Reasoning similarly for $au$ in $\O_1$ and using the Moser-Trudinger inequality \eqref{m-t} we obtain
\begin{align} \label{mt1}
\begin{split}
	8\pi \log \int_M e^u\,dV_g + \rho_{2,a} \log \int_M e^{au}\,dV_g  \leq &~ \frac 12 \int_M |\n(\chi_1 w)|^2 \,dV_g+ C\|v\|_{L^{\infty}(M)}\\
	  & + (8\pi+a\rho_{2,a})\int_M \chi_1 w \,dV_g + C.
\end{split}
\end{align}
By the Poincar\'e's and Young's inequalities it is possible to show
\begin{equation} \label{mt2}
 \int_M \chi_1 w \,dV_g  \leq \e \int_M |\n w|^2 \,dV_g + C_{\e}.
\end{equation}
while the gradient term can be estimated by the Young's inequality as follows:
\begin{align}
	\int_M |\n(\chi_1 w)|^2 \,dV_g \leq (1+\e)\int_{B_{\ov\d/2}(\O_1)}  |\n w|^2 \,dV_g + C_{\e,\ov\d}\int_M w^2 \,dV_g. \label{mt3}
\end{align}
For $m=2,\dots,k$ we use the spreading of $u$ and the standard Moser-Trudinger inequality \eqref{ineq} and the same argument as above to get
\begin{align} \label{mt-f2}
\begin{split}
8\pi \log \int_M e^u\,dV_g  \leq ~&\frac{1+\e}{2}\int_{B_{\ov\d/2}(\O_m)}  |\n w|^2 \,dV_g + \e \int_M |\n w|^2 \\
&+ C_{\e,\ov\d}\int_M w^2 \,dV_g + C\|v\|_{L^{\infty}(M)}+ C.
\end{split}
\end{align}
By combining \eqref{mt1}, \eqref{mt-f2} and recalling that the sets $\O_j$ are disjoint we end up with
\begin{align}
8k\pi \log \int_M e^u\,dV_g + \rho_{2,a} \log \int_M e^{au}\,dV_g  \leq ~&\frac{1+\e}{2}\int_M  |\n w|^2 \,dV_g + \e \int_M |\n w|^2 \,dV_g \nonumber\\
&+ C_{\e,\ov\d}\int_M w^2 \,dV_g + C\|v\|_{L^{\infty}(M)}+ C. \label{mt-f3}
\end{align}
Finally, we can suitably choose $v,w$ by means of a decomposition of $u$ relative to a basis of eigenfunctions of $-\D$ in $H^1(M)$ with zero average condition to estimate the left terms, see for example \cite{je-ya, mal-ndi}.
\end{proof}

\medskip

Proposition \ref{mt-impr} tells us that in the case when $J_\rho$ is sufficiently negative, the functions $e^u$ can not be spread over the surface (otherwise we would have a lower bound on $J_\rho$) and hence they are close to the set $M_k$ in \eqref{M_k}: more precisely, recalling $J_\rho^L$ in \eqref{sub} and $\dkr$ in \eqref{dist}, it holds the following (see for example \cite{zhou} for the first part of the statement and \cite{bjmr} for the second one).
\begin{pro}\label{p:map}
Suppose $\dis{\rh_1\in(8k\pi,8(k+1)\pi)}$, $k\in\N$ and
$\dis{\rh_2<\rho_{2,a}}$. Then, for any
$\displaystyle{\e>0}$ there exists $\dis{L>0}$ such that if
$\dis{u\in J_\rh^{-L}}$ then
$$\dkr\left(\frac{\
h_1e^{u}}{\int_M
h_1e^{u}\,dV_g},M_k\right)<\e.
$$
Moreover, for $L$ sufficiently large there exists a continuous retraction
$$
	\Psi: J_\rho^{-L} \to M_k.
$$
\end{pro}

\subsection{Min-max scheme}

We construct now the min-max scheme which yields existence of solutions to \eqref{eq}. We observed in the previous subsection that the low sublevels $J_\rho^{-L} $ are mapped to the set $M_k$, see Proposition \ref{p:map}. The min-max scheme will be based on the latter set: to this end we need the following key result which asserts that $M_k$ and the associated map $\Psi$ are a good description of $J_\rho^{-L}$, in the sense that one can define a reverse map such that the composition of them is homotopic to the identity map, see for example \cite{bjmr,mal-ndi,zhou}. Recall the map $\Psi$ in Proposition \ref{p:map} and that $\rho_{2,a}=\dfrac{8\pi}{a^2}-\dfrac{16\pi}{a}$.
\begin{pro} \label{p:test}
Suppose $\dis{\rh_1\in(8k\pi,8(k+1)\pi)}$, $k\in\N$ and $\dis{\rh_2<\rho_{2,a}}$. Then, for $L$ sufficiently large there exists a continuous map
$$
	\Phi: M_k \to J_\rho^{-L},
$$
such that $\Psi\circ\Phi\cong Id_{M_k}$.
\end{pro}

\begin{proof}
We sketch the proof for the reader's convenience. For large $L$ and the parameter $\l>0$ to be defined in the sequel we consider $\s := \sum_{i=1}^k t_i \d_{x_i} \in M_k$,
\begin{equation} \label{bubble}
    \var_{\l , \s} (x) =   \log \, \sum_{i=1}^{k} t_i \left( \frac{1}{1 + \l^2 d(x,x_i)^2} \right)^2,
\end{equation}
and define $\Phi=\Phi_\l: M_k \to H^1(M)$ by $\Phi_\l(\s) = \var_{\l , \s}$. We need to show
\begin{equation} \label{claim}
  J_{\rho}(\var_{\l , \s}) \to - \infty \quad \hbox{ as  } \l \to + \infty,
  \qquad \quad \hbox{ uniformly in }  \s \in M_k.
\end{equation}
Concerning the gradient part we claim that
\begin{equation}\label{grad}
\frac{1}{2} \int_{M} |\n \var_{\l , \s}(x)|^2 \,dV_g \leq \bigr(16 k \pi+o_\l(1)\bigr) \log \l + C.
\end{equation}
It is indeed possible to show the following two estimates:
\begin{equation} \label{gr1}
    |\n \var_{\l , \s}(x)| \leq C \l, \qquad \mbox{for every $x\in M$},
\end{equation}
where $C$ is a constant independent of $\l$, $\s \in M_k$, and
\begin{equation} \label{gr2}
    |\n \var_{\l , \s}(x)| \leq \frac{4}{d_{min}(x)}, \qquad \mbox{for every $x\in M,$}
\end{equation}
where $\dis{d_{min}(x) = \min_{i=1,\dots,k} d(x,x_i)}$.

We consider then
\begin{align*}
    \frac{1}{2} \int_{M} |\n \var_{\l , \s}(x)|^2 \,dV_g =~& \frac{1}{2} \int_{\bigcup_i B_{\frac{1}{\l}}(x_i)} |\n \var_{\l , \s}(x)|^2 \,dV_g \\
    &+ \frac{1}{2}\int_{M \setminus \bigcup_i B_{\frac{1}{\l}}(x_i)} |\n \var_{\l , \s}(x)|^2 \,dV_g.
\end{align*}
From (\ref{gr1}) we get
$$
    \int_{\bigcup_i B_{\frac{1}{\l}}(x_i)} |\n \var_{\l , \s}(x)|^2 \,dV_g \leq C.
$$
We then introduce the sets
$$
    A_i = \left\{ x \in M : d(x,x_i) = \min_{j=1,\dots,k} d(x,x_j) \right\},
$$
and by (\ref{gr2}) we obtain
\begin{align*}
    \frac 12 \int_{M \setminus \bigcup_i B_{\frac{1}{\l}}(x_i)} |\n \var_{\l , \s}(x)|^2 \,dV_g
     \leq ~& 8 \sum_{i=1}^k \int_{A_i \setminus B_{\frac{1}{\l}}(x_i)} \frac{1}{d_{min}^2(x)} \,dV_g + C \\
                                         \leq~& \big(16 k \pi+o_\l(1)\big) \log \l + C.
\end{align*}

For the part involving the nonlinear term we claim
\begin{equation} \label{exp1}
\log \int_M  e^{\var_{\l , \s}} \,dV_g =  - 2 \log \l + O(1).
\end{equation}
It is enough to estimate
$$
\int_M \frac{1}{\bigr( 1 + \l^2 d(x,\ov{x})^2 \bigr)^2} \,dV_g,
$$
for some fixed $\ov x\in M$. By a change of variables it is easy to show that
$$
\int_M \frac{1}{\bigr( 1 + \l^2 d(x,\ov{x})^2 \bigr)^2} \,dV_g= \l^{-2}(1 + O(1)),
$$
which gives the claim in \eqref{exp1}.

Concerning the average part we claim that
\begin{equation} \label{aver}
	\int_M {\var_{\l , \s}} \, dV_g = -\bigr(4+o_\l(1)\bigr) \log \l + O(1).
\end{equation}
For simplicity we show the latter estimate just for $k=1$. It holds
$$
	{\var_{\l , \s}}(x) = -4 \log \bigr( \max\{ 1,\l d(x, x_1) \} \bigr) + O(1), \qquad x_1 \in M.
$$
We have
\begin{align*}
	\int_M {\var_{\l , \s}}\, dV_g &= -4 \int_{M \setminus B_{\frac{1}{\l}}(x_1)} \log \bigr(\l d(x, x_1)\bigr) \,dV_g - 4 \int_{B_{\frac{1}{\l}}(x_1)}\, dV_g + O(1) \\
	&= -4 \log\l \left|M \setminus B_{\frac{1}{\l}}(x_1)\right| -4 \int_{M \setminus B_{\frac{1}{\l}}(x_1)} \log (d(x, x_1)) \,dV_g + O(1).
\end{align*}
Recalling that we have $|M|=1$ the claim \eqref{aver} is proved.

Finally, using first the Jensen's inequality involving the part $e^{a(\var_{\l , \s}-\ov \var_{\l , \s})}$ and then \eqref{grad}, \eqref{exp1}, \eqref{aver} we deduce
\begin{align*}
	J_\rho(\var_{\l , \s}) =~&\frac{1}{2}\int_M |\nabla \var_{\l , \s}|^2 \,dV_g - \rho_1  \left( \log\int_M h_1  e^{\var_{\l , \s}} \,dV_g  - \int_M \var_{\l , \s} \,dV_g \right)  \\
	&  -  {\rho_2}  \left( \log\int_M h_2  e^{a \var_{\l , \s}} \,dV_g  - \int_M  a \var_{\l , \s} \,dV_g \right) \\
			\leq~&\frac{1}{2}\int_M |\nabla \var_{\l , \s}|^2 \,dV_g - \rho_1  \left( \log\int_M h_1  e^{\var_{\l , \s}} \,dV_g  - \int_M \var_{\l , \s} \,dV_g \right)\\
			\leq~&\bigr(16k\pi-2\rho_1+o_\l(1)\bigr) \log\l + O(1),
\end{align*}
which proves \eqref{claim} since by assumption $\rho_1>8k\pi$.

We prove now the final assertion of the proposition. It is standard to see that
$$
\frac{e^{\var_{\l , \s}}}{\int_M e^{\var_{\l , \s}}\,dV_g} \rightharpoonup \s, \qquad \mbox{as } \l\to+\infty,
$$
in the sense of measures, see for example \cite{bjmr,mal-ndi,zhou}. We point out that the map $\Psi$ in Proposition \ref{p:map} is a retraction: it follows then that
$$
\Psi\left(\frac{e^{\var_{\l , \s}}}{\int_M e^{\var_{\l , \s}}\,dV_g}\right) \to \s,
$$
which gives the desired homotopy $\Psi\circ\Phi\cong Id_{M_k}$.
\end{proof}

\medskip

We are now in the position to construct the min-max scheme and to prove the Theorem \ref{th:exist}.

\medskip

\begin{proof}[Proof of Theorem \ref{th:exist}.]
Let $\overline{M}_k$ be the topological cone over $M_k$, i.e.
\begin{equation} \label{cone}
   \overline{M}_k = \bigr(M_k \times [0,1]\bigr) \Bigr/ \bigr(M_k \times \{ 1 \}\bigr),
\end{equation}
where the equivalence relation identifies all the points in $M_k \times \{ 1 \}$. We choose $L > 0$ so large that the map $\Psi: J_\rho^{-L} \to M_k$ in Proposition \ref{p:map} is well-defined and let $\l>1$ be so large such that $J_\rho(\varphi_{\l, \s}) \leq - 4L$ uniformly in $\s\in M_k$, where $\varphi_{\l, \s}$ is given in \eqref{bubble}. We then consider the following family of functions
\begin{equation}\label{class}
    \mathcal{H} = \Bigr\{ h : \overline{M}_k \rightarrow H^1(\Sigma) \; : \; h \hbox{ is continuous and } h\bigr(\s,0\bigr)
 = \varphi_{\l,\s} \,,\, \s\in M_k \Bigr\},
\end{equation}
which clearly is non-empty: indeed, $\bar{h}(\s,s) = (1-s) \,\varphi_{\l, \s}$, $s\in[0,1],\,\s\in M_k$ belongs to $\mathcal{H}$. Letting
$$
  c_{\rho} = \inf_{h \in \mathcal{H}}\; \sup_{m \in \overline{M}_k} J_{\rho}\bigr(h(m)\bigr),
$$
the key property is that $$c_{\rho} > - 2 L.$$ Indeed, assuming by contradiction that $c_{\rho}\leq -2L$ and there would  exist a map $h \in \mathcal{H}$ with $\sup_{m \in \overline{M}_k} J_{\rho}\bigr(h(m)\bigr) \leq - L$. Then, let $m = (\s, t)$, with $\vartheta \in M_k$, from Proposition \ref{p:map} we can conclude the map $ t \mapsto \Psi \circ h(\cdot,t)$ would be a homotopy in $M_k$ between $ \Psi \circ \varphi_{\l,\s}$ (for $t=0$) and a constant map (for $t=1$ due to the equivalence relation in \eqref{cone}). But this is impossible since $M_k$ is non-contractible and since $\Psi \circ \varphi_{\l,\s}=\Psi\circ\Phi\cong Id_{M_k}$ by Proposition \ref{p:test}. Therefore we deduce the desired estimate.

On the other hand, by construction and by the choice of $\l>1$ we have
$$
	\sup_{h\in\mathcal H}\sup_{m\in\partial \ov{M}_k} J_\rho\bigr(h(m)\bigr) = \sup_{\s\in M_k} J_\rho(\varphi_{\l, \s}) \leq -4L.
$$

We conclude that the functional $J_{\rho}$ has a min-max structure which in turn yields a Palais-Smale sequence. However, we cannot directly conclude the existence of a critical point, since it is not known whether the Palais-Smale condition holds or not. It is then standard to follow the \emph{monotonicity argument} introduced in \cite{struwe} to get in a first step bounded Palais-Smale sequences and then a sequence of solutions to \eqref{eq} relative to $\rho_{1,k}\to \rho_1\in (8k\pi,8(k+1)\pi)$, $\rho_{2,k}\to\rho_2<\rho_{2,a}$. It is then sufficient to apply the compactness result in Theorem \ref{th:comp} to deduce convergence to a solution of \eqref{eq}.
\end{proof}

\

\end{document}